\documentclass{article}
\usepackage[utf8]{inputenc}
\usepackage[margin = 1.2in]{geometry}
\usepackage{amsmath}
\usepackage{amssymb}
\usepackage{amsthm}
\usepackage{bbm}
\usepackage[english]{babel}
\usepackage{fancyhdr}
\usepackage[colorlinks=true, allcolors=blue]{hyperref}
\usepackage{stmaryrd}
 
\newtheorem{theorem}{Theorem}[section]

\newtheorem{lemma}[theorem]{Lemma}

\newtheorem{prop}{Proposition}[section]

\theoremstyle{definition}
\newtheorem{definition}{Definition}[section]

\DeclareMathOperator{\rk}{rk}
\DeclareMathOperator{\ad}{ad}
\DeclareMathOperator{\Ad}{Ad}
\DeclareMathOperator{\coad}{ad^{\ast}}
\DeclareMathOperator{\coAd}{Ad^{\ast}}
\DeclareMathOperator{\Gr}{Gr}
\DeclareMathOperator{\fg}{\mathfrak g}
\DeclareMathOperator{\fgd}{\mathfrak g^{\ast}}
\DeclareMathOperator{\fd}{\mathfrak d}
\DeclareMathOperator{\fb}{\mathfrak b}
\DeclareMathOperator{\fh}{\mathfrak h}
\DeclareMathOperator{\fp}{\mathfrak p}
\DeclareMathOperator{\fa}{\mathfrak a}
\DeclareMathOperator{\fad}{\mathfrak a^{\ast}}
\DeclareMathOperator{\fl}{\mathfrak l}

\title{\bf {Abelian Ideals and the Variety of Lagrangian Subalgebras}}
\author{Sam Evens and Yu Li}
\date{\vspace{-5ex}}

\begin{document}

\maketitle

\begin{abstract}
    For a semisimple algebraic group $G$ of adjoint type with Lie algebra $\fg$ over the complex numbers, we establish a bijection between the set of closed orbits of the group $G \ltimes \fgd$ acting on the variety of Lagrangian subalgebras of $\fg \ltimes \fgd$ and the set of abelian ideals of a fixed Borel subalgebra of $\fg$.  In particular, the number of such orbits equals $2^{\text{rk} \fg}$ by Peterson's theorem on abelian ideals.
\end{abstract}

\tableofcontents

\medskip

\section{Introduction}

\par In the 1990's, Dale Peterson proved the remarkable and surprising assertion that the number of abelian ideals of a Borel subalgebra of a semisimple Lie algebra is $2^l$, where $l$ is the rank of the Lie algebra.  Although Peterson never published his result, his argument was widely popularized by Kostant, who further developed the theory in \cite{Kos2}, and the theory was further developed by many others, including in \cite{CP0, CP, CP2, Pan0, PR, Sut}.   In this paper, we find a new connection between the theory of \textit{abelian ideals} and the \textit{variety of Lagrangian subalgebras}, which is a variety that arises naturally in Poisson geometry.

Throughout this paper, we work over the complex numbers $\mathbb C$.   Let $G$ be a semisimple algebraic group of adjoint type with Lie algebra $\fg$.  There is a bi-vector field $\pi_{st}$, known as the {\it standard Poisson bi-vector field}, on $G$, making it a {\it Poisson Lie group}.  Infinitesimal information of $(G, \pi_{st})$ near the identity element of the group $G$ is captured by the {\it Manin triple} $(\fg \oplus \fg, \fg_{\Delta}, \fg^{\ast}_{st})$ (c.f. \cite{EvLu, EvLu2}).  Here, the symbol $\fg_{\Delta}$ stands for the diagonal Lie subalgebra of $\fg \oplus \fg$ and $\fg^{\ast}_{st}$ is defined to be the Lie subalgebra $$\{(h+x, -h+y): h \in \fh, x \in \mathfrak n, y \in \mathfrak n^-\}$$ of $\fg \oplus \fg$, where $\fg = \mathfrak n \oplus \fh \oplus \mathfrak n^-$ is a triangular decomposition of $\fg$.  The Lie algebra $\fg \oplus \fg$ comes equipped with a symmetric nondegenerate invariant bilinear form $<~,~>$.  A Lie subalgebra $\fl$ of $\fg \oplus \fg$ of dimension $\dim \fg$ such that $<\fl, \fl> = 0$ is called a {\it Lagrangian subalgebra} of $\fg \oplus \fg$.  The variety $\mathcal L(\fg \oplus \fg)$ of Lagrangian subalgebras of $\fg \oplus \fg$ has been extensively studied in algebraic and Poisson geometry:

\begin{itemize}
    \item In \cite{EvLu, EvLu2}, the first author and Lu discussed a $(G \times G)$-action on $\mathcal L(\fg \oplus \fg)$ and described explicitly the $(G \times G)$-orbits and their closures in $\mathcal L(\fg \oplus \fg)$.  Orbits of many subgroups of $G \times G$, for example the diagonal subgroup $G_{\Delta}$, are also explicitly described.
    \item It follows directly from the fundamental work of De Concini and Procesi \cite{DCP} (see \cite{EvJ, EvLu2}) that there is an embedding of $G$ into $\mathcal L(\fg \oplus \fg)$ (as a single $(G \times G)$-orbit) so that the closure of the image of this embedding is isomorphic to the {\it wonderful compactification} of $G$.
    \item The variety $\mathcal L(\fg \oplus \fg)$ can be thought of as a `universal {\it Poisson homogeneous space}' for the Poisson Lie group $(G, \pi_{st})$ as well as its {\it Poisson dual group} (\cite{Dri2, EvLu, EvLu2}).
\end{itemize}

The zero bi-vector field on $G$ makes $G$ into a Poisson Lie group.  The corresponding Manin triple is $(\fg \ltimes \fgd, \fg, \fgd)$.  Following the construction from Section 2 of \cite{EvLu}, we consider the variety $\mathcal L = \mathcal L(\fg \ltimes \fgd)$, called the {\it variety of Lagrangian subalgebras} of $\fg \ltimes \fgd$.  We note that the construction of $\mathcal L(\fg \ltimes \fgd)$ is the analogue of the construction of $\mathcal L(\fg \oplus \fg)$ when we substitute the Manin triple $(\fg \ltimes \fgd, \fg, \fgd)$ in place of $(\fg \oplus \fg, \fg_{\Delta}, \fg^{\ast}_{st})$.  We view $\fgd$ as an additive algebraic group and let $D$ denote the semidirect product $G \ltimes \fgd$, viewed as an algebraic group.  There is an action of $D$ on $\mathcal L$ analogous to the $(G \times G)$-action on $\mathcal L(\fg \oplus \fg)$, and the purpose of this paper is to classify closed $D$-orbits in $\mathcal L$.

A Lie ideal $\fa$ of a Borel subalgebra of $\fg$ such that $[\fa, \fa] = 0$ is called an {\it abelian ideal} of the Borel subalgebra.  The main result of this paper is the following

\begin{theorem} \label{mainthm}
There is a bijection between the set of closed $D$-orbits in $\mathcal L$ and the set of abelian ideals of a fixed Borel subalgebra of $\fg$.
\end{theorem}

We will make this bijection explicit after introducing some notation.  Along the way, we will see that the geometry of $\mathcal L$ is much harder to understand than that of $\mathcal L(\fg \oplus \fg)$.  The subgroup $G \cong G \times \{0\}$ of $D$ is an analogue of the subgroup $G_{\Delta}$ of $G \times G$.  We will see that a complete understanding of the orbits of $G$ in $\mathcal L$ contains as a subquestion the question of classifying all finite dimensional Lie algebras.  So our current situation is quite different than that of \cite{EvLu, EvLu2}.  It is for this reason that we restricted our attention to closed $D$-orbits in $\mathcal L$.  In a subsequent publication, we will study more general $D$-orbits and orbits of certain subgroups of $D$, using a degeneration of $\mathcal L (\fg \oplus \fg)$ into $\mathcal L$.

Peterson proved the remarkable result that the number of abelian ideals of a fixed Borel subalgebra of $\fg$ equals $2^{\rk \fg}$.  As a consequence of Theorem \ref{mainthm} and Peterson's result, we have the following

\begin{theorem}
The number of closed $D$-orbits in $\mathcal L$ is exactly $2^{\rk \fg}$.
\end{theorem}

Denote by $(G, 0)$ the Poisson Lie group $G$ equipped with the zero Poisson bi-vector field.  We note that the Poisson dual group of $(G, 0)$ is $(\fgd, \pi_{KK})$, where $\pi_{KK}$ stands for the {\it Kirillov-Kostant Poisson bi-vector field} on $\fgd$.  In our context, the Poisson Lie groups $(G,0)$ and $(\fgd, \pi_{KK})$ play a similar role as that of $(G, \pi_{st})$ and its Poisson dual group in \cite{EvLu, EvLu2}.  
 The theory of abelian ideals of Borel subalgebras has been initiated by Kostant in \cite{Kos2}, and it is a well developed theory by now, c.f. \cite{CP0, CP, CP2, Pan0, PR, Sut}.  However, the proofs of several key results of this theory are combinatorial in nature.  We hope that Poisson geometry may provide an alternative understanding of these results.

In the paper \cite{EvLu3}, a Poisson structure $\pi$ is introduced on $G$ for which the symplectic leaves are related to conjugacy classes and the \textit{Bruhat decomposition}.
In a rough sense, it is reasonable to regard $(\fgd, \pi_{KK})$, resp. $\mathcal L$, as an additive analogue of $(G, \pi)$, resp. $\mathcal L(\fg \oplus \fg)$.  Since many important geometric objects associated to $G$ can be realized as closed subvarieties of $\mathcal L(\fg \oplus \fg)$, c.f. \cite{EvJ, EvLu, EvLu2}, one can define these objects in the additive setting as well.  For example, one can study the geometry of the `wonderful compactification' of $\fgd$, the closure of a single $D$-orbit in $\mathcal L$, and of a Cartan subalgebra of $\fg$.  The latter is a counterpart of the wonderful compactification of a Cartan subgroup of $G$ and potentially leads to a theory of \textit{additive toric varieties}.  These topics will be discussed in a separate paper.

This paper is organized as follows.  In Section 2 we introduce notation and define various basic objects.  In Section 3, we construct the variety $\mathcal L$ of Lagrangian subalgebras of $\fg \ltimes \fgd$.  We recall the algebraic parametrization of points of $\mathcal L$ due to Karolinsky and Stolin \cite{KS}.  In Section 4 we prove Theorem \ref{mainthm}.  Our proof is based on the Karolinsky-Stolin parametrization.

{\bf Acknowledgements.}  The authors are grateful to Victor Ginzburg and Jiang-Hua Lu for many stimulating discussions.   We also thank Eugene Karolinsky for useful discussions.   The first author was supported in part by the Simons Foundation Travel Grant 359424.

\medskip

\section{Notation and Preliminaries}

Let $G$ be a semisimple algebraic group of adjoint type with Lie algebra $\fg$.  By $\fgd$ we mean the dual space of $\fg$.  In what follows, $\fgd$ will also be viewed as an algebraic group and an abelian Lie algebra, where the group operation is the usual addition in $\fgd$.  The adjoint (resp. coadjoint) action of $G$ on $\fg$ (resp. $\fgd$) will be denoted by $\Ad$ (resp. $\coAd)$.  The adjoint (resp. coadjoint) action of $\fg$ on $\fg$ (resp. $\fgd$) will be denoted by $\ad$ (resp. $\coad)$.  Elements of $\fg$ will be denoted by Roman letters $x, y, \dots$ and elements of $\fgd$ will be denoted by Greek letters $\alpha, \beta, \dots$.

We consider the semidirect product algebraic group $D := G \ltimes \fgd$, which has the following structure.  As a set, $D$ is the Cartesian product $G \times \fgd$.  Multiplication and inversion in $D$ are given by
\begin{equation} \label{gp}
    \begin{aligned}
    (g, \alpha)(g', \alpha') := & (gg', \Ad^{\ast}_{(g')^{-1}} \alpha + \alpha') ~ && \text{for all} ~ g, g' \in G ~ \text{and} ~ \alpha, \alpha' \in \fgd \\
    (g, \alpha)^{-1} := & (g^{-1}, - \Ad^{\ast}_g \alpha) ~ && \text{for all} ~ g \in G ~ \text{and} ~ \alpha \in \fgd.
    \end{aligned}
\end{equation}
The identity element of $D$ is $(e,0)$, where $e$ stands for the identity element of $G$.  It is easy to verify that the homomorphism of algebraic groups $G \rightarrow D$ (resp. $\fgd \rightarrow D$) given by $g \mapsto (g, 0)$ (resp. $\alpha \mapsto (e, \alpha)$) is an embedding of $G$ (resp. $\fgd$) into $D$.  In what follows, we will view $G$ (resp. $\fgd$) as an algebraic subgroup of $D$ via this homomorphism.

Consider the Lie algebra $\fd$ of $D$, which is the semidirect sum Lie algebra $\fd = \fg \ltimes \fgd$.  As a vector space, it is the Cartesian product $\fg \times \fgd$.  The Lie bracket in $\fd$ is given by
\begin{align*}
    [(x, \alpha), (y, \beta)] := ([x,y], \ad^{\ast}_x \beta - \ad^{\ast}_y \alpha) ~ && \text{for all} ~ x, y \in \fg ~ \text{and} ~ \alpha, \beta \in \fgd.
\end{align*}
As above, $\fg$ (resp. $\fgd$) will be viewed as a Lie subalgebra of $\fd$ via the homomorphism $\fg \rightarrow \fd$ (resp. $\fgd \rightarrow \fd$) given by $x \mapsto (x, 0)$ (resp. $\alpha \mapsto (0, \alpha)$).  The exponential map $\text{Exp}: \fd \rightarrow D$ is given by
\begin{align} \label{exp}
    (x, \alpha) \mapsto (\exp(x), \alpha),
\end{align}
where $\exp: \fg \rightarrow G$ stands for the exponential map for $G$. 

By abuse of notation, the adjoint action of $D$ on $\fd$ will also be denoted by $\Ad$.  Using formulas (\ref{gp}) and (\ref{exp}), one easily verifies that
\begin{align*}
    \Ad_{(g, \alpha)} (x, \beta) = (\Ad_g x, - \Ad^{\ast}_g \ad^{\ast}_x \alpha + \Ad^{\ast}_g \beta) ~ && \text{for all} ~ g \in G, x \in \fg ~ \text{and} ~ \alpha, \beta \in \fgd.
\end{align*}
In particular, for $g \in G, x \in \fg$ and $\alpha, \beta \in \fgd$, we have
\begin{align} \label{Ads}
    \begin{aligned}
    \Ad_g (x, \beta) = (Ad_g x, \Ad^{\ast}_g\beta) \\
    \Ad_{\alpha} (x, \beta) = (x, - \ad^{\ast}_x \alpha + \beta).
    \end{aligned}
\end{align}

Define a bilinear form $(~,~): \fd \otimes \fd \rightarrow \mathbb C$ by $$((x, \alpha), (y, \beta)) := \alpha (y) + \beta (x),$$ for all $x, y \in \fg$ and $\alpha, \beta \in \fgd$.  It is easily verified that $(~,~)$ is symmetric, nondegenerate and $D$-invariant.  Here, $D$-invariance means
\begin{align} \label{inv}
    (\Ad_{(g, x)} (y, \beta), \Ad_{(g, x)} (z, \gamma)) = ((y, \beta), (z, \gamma)),
\end{align}
for all $g \in G$, $x,y,z \in \fg$ and $\beta, \gamma \in \fgd$.

Recall that for a vector space $V$ equipped with a symmetric nondegenerate bilinear form $(~,~)$, a vector subspace $W$ of $V$ is called {\it isotropic} if $(w, w') = 0$ for all $w, w' \in W$.  An isotropic subspace $W$ of $V$ is called {\it Lagrangian} if $W$ is maximal, with respect to inclusion, among all isotropic subspaces of $V$.  When $V$ is even dimensional, an isotropic subspace $W$ of $V$ is Lagrangian if and only if $\dim W = \frac 12 \dim V$.

A Lie subalgebra $\fl$ of $\fd$ is called a {\it Lagrangian subalgebra} if $\fl$ is a Lagrangian vector subspace of $\fd$ with repsect to the bilinear form on $\fd$ introduced above.  Although we will not use it in this paper, it is worth pointing out that $(\fd, \fg, \fgd)$, together with the bilinear form $(~,~)$ on $\fd$, is a Manin triple.  This amounts to saying that $\fg$ and $\fgd$ are Lagrangian subalgebras of $\fd$ and $\fd$ is the {\it direct} sum of $\fg$ and $\fgd$ as vector spaces, c.f. \cite{Dri2, EvLu, EvLu2} for details.

\medskip

\section{The Variety of Lagrangian Subalgebras}



Let $\fa$ be a Lie subalgebra of $\fg$ and $\fad$ its dual space.  We write $Z^1_{CE}(\fa, \fad)$ for the vector space of {\it Chevalley-Eilenberg $1$-cocycles} on $\fa$ with coefficients in $\fad$, where $\fa$ acts on $\fad$ by the coadjoint action.  Specifically, $Z^1_{CE}(\fa, \fad)$ consists of linear maps $f: \fa \rightarrow \fad$ such that
\begin{align} \label{ce}
    \ad^{\ast}_x f(y) - \ad^{\ast}_y f(x) - f([x,y]) = 0 ~ && \text{for all} ~ x,y \in \fa.
\end{align}
We say that $f \in Z^1_{CE}(\fa, \fad)$ is {\it skew} if 
\begin{align} \label{skew}
    f(x)(y) + f(y)(x) = 0 ~ && \text{for all} ~ x,y \in \fa.
\end{align}

\begin{definition} \label{lg}
Let $\fa$ be a Lie subalgebra of $\fg$ and $f \in Z^1_{CE}(\fa, \fad)$ a skew element.  We define
\begin{align*}
    \fl(\fa, f) := \{(x, \alpha) \in \fd: x \in \fa, f(x) = \alpha|_{\fa}\},
\end{align*}
where $\alpha|_{\fa}$ stands for the restriction of $\alpha$ to $\fa$.
\end{definition}

Let $\fa$ and $f$ be as in Definition \ref{lg}.  It follows from (\ref{ce}) (resp. (\ref{skew})) that $\fl(\fa, f)$ is a Lie subalgebra (resp. isotropic vector subspace) of $\fd$.  It is clear that $\dim \fl(\fa, f) = \frac 12 \dim \fd$.  Therefore, $\fl(\fa, f)$ is a Lagrangian subalgebra of $\fd$.

The following result of Karolinsky and Stolin establishes a bijection between points of $\mathcal L$ and pairs $(\fa, f)$ as in Definition \ref{lg}.

\begin{theorem} \cite{KS}
The Lagrangian subalgebras of $\fd$ are exactly the Lagrangian subalgebras $\fl(\fa, f)$, for $(\fa, f)$ as in Definition \ref{lg}.
\end{theorem}

Let $(\fa, f)$ and $(\fa', f')$ be as in Definition \ref{lg}.  It is easily verified that $\fl(\fa, f) = \fl(\fa', f')$ if and only if $\fa = \fa'$ and $f = f'$.

Let $\Gr(n, \fd)$ be the {\it Grassmannian} of $n$-dimensional vector subspaces of $\fd$, where $n := \dim \fg = \frac 12 \dim \fd$.  For an $n$-dimensional vector subspace of $\fd$, the property of being a Lie subalgebra, resp. being an isotropic vector subspace, of $\fd$ is a closed condition.  Therefore, the set of Lagrangian subalgebras of $\fd$ has a natural structure of a reduced closed subvariety of $\Gr(n, \fd)$, to be denoted by $\mathcal L$.  Since $\Gr(n, \fd)$ is a projective variety, the variety $\mathcal L$ is projective as well.  We will call $\mathcal L$ the {\it variety of Lagrangian subalgebras} of $\fd$.


Let $S$ be a subset of $\fd$.  For any $d \in D$, we write $\Ad_d S$ for the set $\{\Ad_d s: s \in S\}$.  Since the $D$-action on $\fd$ preserves the symmetric form $(~,~)$, it follows that if $\fl$ is a Lagrangian subalgebra of $\fd$, then so is $\Ad_d \fl$ for any $d \in D$.  Therefore, we obtain an action, still denoted by $\Ad$, of $D$ on $\mathcal L$.

We deduce some formulas for the action of $D$ on $\mathcal L$.  Let $\fa$ be a Lie subalgebra of $\fg$ and $f \in Z^1_{CE}(\fa, \fad)$ a skew element.  For $g \in G$, we define a linear map $g.f: \Ad_g \fa \rightarrow (\Ad_g \fa)^{\ast}$ by $$(g.f)(x)(y) := f(\Ad_{g^{-1}} x)(\Ad_{g^{-1}} y)$$ for all $x,y \in \Ad_g \fa$.  A simple computation shows that $g.f$ is a skew element of $Z^1_{CE}(\Ad_g \fa, (\Ad_g \fa)^{\ast})$.  For $\alpha \in \fgd$, we define a linear map $f_{\alpha}: \fa \rightarrow \fad$ by $$f_{\alpha}(x) := - \ad^{\ast}_x (\alpha|_{\fa})$$ for all $x \in \fa$.  Another simple computation shows that $f_{\alpha}$ is a skew element of $Z^1_{CE}(\fa, \fad)$.

\begin{lemma} \label{halfact}
Let $\fa$ be a Lie subalgebra of $\fg$ and $f \in Z^1_{CE}(\fa, \fad)$ a skew element.  Then,
\begin{enumerate}
    \item For any $g \in G$, we have $$\Ad_g \fl(\fa, f) = \fl(\Ad_g \fa, g.f).$$
    \item For any $\alpha \in \fgd$, we have $$\Ad_{\alpha} \fl(\fa, f) = \fl(\fa, f+f_{\alpha}).$$
\end{enumerate}
\end{lemma}

\begin{proof}
We only prove the first statement.  The proof of the second statement is analogous.

Note that $\Ad_g \fl(\fa, f)$ and $\fl(\Ad_g \fa, g.f)$ are $n$-dimensional vector subspaces of $\fd$, hence it suffices to show that $\Ad_g \fl(\fa, f) \subseteq \fl(\Ad_g \fa, g.f)$.  Let $(x, \beta) \in \fl(\fa, f)$.  Then we have $$\Ad_g (x, \beta) = (\Ad_g x, \Ad^{\ast}_g \beta)$$ by formula (\ref{Ads}).  For any $y \in \fa$, we have $$(g.f)(\Ad_g x)(\Ad_g y) = f(x)(y) = \beta (y) = (\Ad^{\ast}_g \beta)(\Ad_g y).$$  Hence, we have $$(g.f)(\Ad_g x) = (\Ad^{\ast}_g \beta)|_{\Ad_g \fa},$$ i.e., $(\Ad_g x, \Ad^{\ast}_g \beta) \in \fl(\Ad_g \fa, g.f)$, proving the desired inclusion.
\end{proof}


{\bf Remark.} Let $\fa$ be a Lie subalgebra of $\fg$.  For any element $g \in G$, by Lemma \ref{halfact}, we have $$\Ad_g \fl(\fa, 0) = \fl(\Ad_g \fa, 0).$$  Thus the $G$-orbit through $\fl(\fa, 0)$ is $$\{\fl(\Ad_g \fa, 0): g \in G\}.$$  It follows from this and the observation that $\fl(\fa, f) = \fl(\fa', f')$ if and only if $\fa = \fa'$ and $f = f'$, that, for any Lie subalgebras $\fa, \fa'$ of $\fg$, the Lagrangian subalgebras $\fl(\fa, 0)$ and $\fl(\fa', 0)$ are in the same $G$-orbit if and only if $\fa$ and $\fa'$ are conjugate by an element of $G$.  Consequently, if we would like to classify all $G$-orbits in $\mathcal L$, we have to first classify Lie subalgebras of $\fg$ up to conjugation by elements of $G$.  By Ado's theorem, every finite dimensional Lie algebra is isomorphic to a Lie subalgebra of $\mathfrak {gl}_N$ for some $N \in \mathbb N$.  So, if we were able to classify $G$-orbits in $\mathcal L$ in the case where $G$ is of type $A$, we would (more or less) be able to classify all finite dimensional Lie algebras.  This shows that any classification of $G$-orbits in $\mathcal L$ is likely very complicated.   In contrast, the classification of $G_{\Delta}$-orbits in 
$\mathcal L(\fg \oplus \fg)$ is given in Theorem 3.7 of \cite{EvLu2}.
\medskip

\section{Closed Orbits}


Below, we use the notation $N_G(\fa)$ (resp. $N_{\fg}(\fa)$) for the normalizer in $G$ (resp. $\fg$) of a Lie subalgebra $\fa$ of $\fg$.

We first prove

\begin{prop} \label{1}
Let $\fa$ be an abelian ideal of a Borel subalgebra $\fb$ of $\fg$.  Then the $D$-orbit $\Ad_D \fl(\fa, 0)$ through $\fl(\fa, 0)$ is closed.
\end{prop}

\begin{proof}
For any $\alpha \in \fgd$, by Lemma \ref{halfact}, we have $$\Ad_{\alpha} \fl(\fa, 0) = \fl(\fa, f_{\alpha}).$$  Since $\fa$ is an abelian Lie algebra, for any $x,y \in \fa$, we have $$f_{\alpha}(x)(y) = (- \ad^{\ast}_x (\alpha|_{\fa}))(y) = \alpha([x,y]) = \alpha (0) = 0.$$  Thus, we see that $f_{\alpha} = 0$, so that $$\Ad_{\alpha} \fl(\fa, 0) = \fl(\fa, 0).$$  It follows from (\ref{gp}) that, for any $(g, \alpha) \in D$, we have $(g, \alpha) = (g, 0)(e, \alpha)$, so $$\Ad_{(g, \alpha)} \fl(\fa, 0) = \Ad_g \Ad_{\alpha} \fl(\fa, 0) = \Ad_g \fl(\fa, 0) = \fl(\Ad_g \fa, 0),$$ again by Lemma \ref{halfact}.  Since $\fl(\Ad_g \fa, 0) = \fl(\fa, 0)$ if and only if $\Ad_g \fa = \fa$, i.e., $g \in N_G(\fa)$, the $D$-orbit $\Ad_D \fl(\fa, 0)$ is isomorphic to $G/N_G(\fa)$ as a variety.  Since $\fa$ is an ideal of $\fb$, the normalizer $N_G(\fa)$ contains the Borel subgroup of $G$ corresponding to $\fb$.  Hence $N_G(\fa)$ is a parabolic subgroup of $G$ and, therefore, $G/N_G(\fa)$ is a projective variety.  It follows that the orbit $\Ad_D \fl(\fa, 0)$ is closed.
\end{proof}

For the converse we need a simple lemma.  Let $\fg = \mathfrak n \oplus \fh \oplus \mathfrak n^-$ be a triangular decomposition of $\fg$.  Write $\Phi = \Phi^+ \sqcup \Phi^-$ for the root system of $(\fg, \fh)$ so that roots in $\Phi^+$ correspond to $\mathfrak n$ and roots in $\Phi^-$ correspond to $\mathfrak n^-$.  For $\lambda \in \Phi$, the $\lambda$-root space is denoted by $\fg_{\lambda}$.  Let $\fp$ be a parabolic subalgebra of $\fg$ containing $\mathfrak n \oplus \fh$ and $\mathfrak i$ a Lie ideal of $\fp$.  Define $\mathfrak i_0$ to be $\mathfrak i \cap \fh$.  Since $\mathfrak i$ is an ideal of $\fp$, $\fp$ acts on $\mathfrak i$ by the adjoint action.  In particular, the Lie subalgebra $\fh$ of $\fp$ acts on $\mathfrak i$ by the adjoint action.  Hence $\mathfrak i$ decomposes into a direct sum of $\fh$-weight spaces.  Since an $\fh$-weight on $\fg$ is either $0$ or a root, we have proved the following

\begin{lemma} \label{rt}
With the above notation, we have $$\mathfrak i = \mathfrak i_0 \oplus \bigoplus \limits_{\lambda \in \Lambda} \fg_{\lambda},$$ where $\Lambda$ is some subset of $\Phi$.
\end{lemma}

Now we are ready to prove

\begin{prop} \label{2}
Let $\fa$ be a Lie subalgebra of $\fg$ and $f \in Z^1_{CE}(\fa, \fad)$ a skew element.  If the $D$-orbit $\Ad_D \fl (\fa, f)$ is closed, then $\fa$ is an abelian ideal of some Borel subalgebra of $\fg$ and $f = 0$.
\end{prop}

\begin{proof}
Write $k$ for the dimension of $\fa$.  Consider the first projection map $\text{pr}_1: \fd \rightarrow \fg$ sending $(x, \alpha)$ to $x$.  For any $(g, \alpha) \in D$, by Lemma \ref{halfact} and formula (\ref{gp}), we have
\begin{align*}
    \Ad_{(g, \alpha)} \fl(\fa, f) = \Ad_g \Ad_{\alpha} \fl(\fa, f) = \Ad_g \fl(\fa, f+f_{\alpha}) = \fl(\Ad_g \fa, g.(f+f_{\alpha})).
\end{align*}
It follows that $\text{pr}_1 (\Ad_{(g, \alpha)} \fl(\fa, f)) = \Ad_g \fa$.  Let $\Gr(k, \fg)$ be the Grassmannian of $k$-dimensional vector subspaces of $\fg$.  From the computation above, we see that $\text{pr}_1$ induces a morphism of varieties $p: \Ad_D \fl (\fa, f) \rightarrow \Gr(k, \fg)$, sending $\Ad_{(g, \alpha)} \fl(\fa, f)$ to $\Ad_g \fa$.

By hypothesis, the orbit $\Ad_D \fl(\fa, f)$ is closed in the projective variety $\mathcal L$.  Hence $\Ad_D \fl(\fa, f)$ is itself a projective variety.  In particular, $\Ad_D \fl(\fa, f)$ is complete.  So, the image $p(\Ad_D \fl(\fa, f))$ of $\Ad_D \fl(\fa, f)$ must also be complete.  But, by the analysis of the previous paragraph, the image $p(\Ad_D \fl(\fa, f))$ is isomorphic as a variety to $G/N_G(\fa)$.  Hence $N_G(\fa)$ is a parabolic subgroup of $G$.

The fiber of $p$ over $\fa \in \Gr(k, \fg)$ is easily seen to be
\begin{align*}
    \{\fl(\fa, g.(f+f_{\alpha})): g \in N_G(\fa), \alpha \in \fgd\}.
\end{align*}
Write $f'$ for $g.(f+f_{\alpha})$ ($g \in N_G(\fa), \alpha \in \fgd$).  Choose a basis $\{v_1, \cdots, v_k\}$ of $\fa$ and extend it to a basis $\{v_1, \cdots, v_k, v_{k+1}, \cdots, v_n\}$ of $\fg$.  Let $\{v_1^{\ast}, \cdots, v_n^{\ast}\}$ be the dual basis of $\fgd$.  For all $1 \le i \le k$, we write $$f'(v_i) = \sum \limits_{j=1}^k a_{ij} (v_j^{\ast}|_{\fa}),$$ for some $a_{ij} \in \mathbb C$.  Hence, by definition, we see that $\fl(\fa, g.(f+f_{\alpha}))$ is spanned by $(v_1, \sum \limits_{j=1}^k a_{1j} v_j^{\ast})$, ... , $(v_k, \sum \limits_{j=1}^k a_{kj} v_j^{\ast})$, $(0, v_{k+1}^{\ast})$, ... , $(0, v_n^{\ast})$.  Recall that, for any basis $\{w_1, \cdots, w_{2n}\}$ of $\fd$, the set $$\{\text{Span} (w_1 + \sum \limits_{j=n+1}^{2n} b_{1j} w_j, \cdots, w_n + \sum \limits_{j=n+1}^{2n} b_{nj} w_j): b_{ij} \in \mathbb C\}$$ is an affine chart on $\Gr(n, \fd)$.  We consider the affine open set given by the basis $\{w_1, \dots, w_{2n}\}$ with  $$\{w_1, \cdots, w_n\} := \{v_1, \cdots, v_k, v^{\ast}_{k+1}, \cdots, v^{\ast}_n\}, ~ \{w_{n+1}, \cdots, w_{2n}\} := \{v_{k+1}, \cdots, v_n, v^{\ast}_1, \cdots, v^{\ast}_k\},$$ and note that the fiber of $p$ over $\fa$ is contained in the corresponding affine open set.  But this fiber is closed in $\mathcal L$, hence it is a projective variety.  Therefore, the fiber of $p$ over $\fa$ must be a finite set.  In particular,
\begin{align*}
    \{\fl(\fa, f+f_{\alpha}): \alpha \in \fgd\}
\end{align*}
is a finite set.  Suppose $f_{\alpha} \neq 0$ for some $\alpha \in \fgd$.  Then we have $f \neq f + f_{\alpha}$, so $\fl(\fa, f) \neq \fl(\fa, f + f_{\alpha})$.  We deduce that the subset $\{\fl(\fa, f+f_{s\alpha}): s \in \mathbb C\}$ is a curve connecting the distinct points $\fl(\fa, f)$ and $\fl(\fa, f+f_{\alpha})$, a contradiction.  Therefore, we have $f_{\alpha} = 0$ for all $\alpha \in \fgd$.  Then, by the definition of $f_{\alpha}$, we have $\alpha ([x,y]) = 0$ for all $\alpha \in \fgd$ and $x,y \in \fa$.  It follows that $[x,y] = 0$ for all $x,y \in \fa$, i.e., $\fa$ is an abelian Lie algebra.

Since $N_G(\fa)$ is a parabolic subgroup of $G$, we see that $N_{\fg}(\fa)$ is a parabolic Lie subalgebra of $\fg$ which contains $\fa$ as a Lie ideal.  Choose a triangular decomposition $\fg = \mathfrak n \oplus \fh \oplus \mathfrak n^-$ of $\fg$ such that $\mathfrak n \oplus \fh \subseteq N_{\fg}(\fa)$.  Then, by Lemma \ref{rt}, we have $$\fa = \fa_0 \oplus \bigoplus \limits_{\lambda \in \Lambda} \fg_{\lambda}$$ for some subset $\Lambda$ of the set $\Phi = \Phi^+ \sqcup \Phi^-$ of roots for $(\fg, \fh)$.  Suppose that there exists $\lambda \in \Phi^+$ such that $-\lambda \in \Lambda$.  Choose an $\mathfrak {sl}_2$-triple $\{e_{\lambda}, f_{\lambda}, h_{\lambda}\}$ such that $e_{\lambda} \in \fg_{\lambda}$, $f_{\lambda} \in \fg_{-\lambda} \subseteq \fa$ and $h_{\lambda} \in \fh$.  Since $\fa$ is a Lie ideal of $N_{\fg}(\fa)$ and $\mathfrak n \subseteq N_{\fg}(\fa)$, we have $h_{\lambda} = [e_{\lambda}, f_{\lambda}] \in \fa$.  But then $[h_{\lambda}, f_{\lambda}] = -2 f_{\lambda} \neq 0$, contradicting the fact that $\fa$ is an abelian Lie algebra.  From this it follows that $\fa \subseteq \mathfrak n \oplus \fh$.  Since $\fa$ is a Lie ideal of $N_{\fg}(\fa)$ and $\mathfrak n \oplus \fh \subseteq N_{\fg}(\fa)$, we conclude that $\fa$ is a Lie ideal of the Borel subalgebra $\mathfrak n \oplus \fh$ of $\fg$.   It is now well-known and very easy to prove that $\fa \subseteq \mathfrak n$ \cite{Kos2}.

Let $H$ be the Cartan subgroup of $G$ corresponding to $\fh$.  It is clear that $H$ normalizes $\fa$, hence is contained in $N_G(\fa)$.  Therefore, the subset
\begin{align*}
    \{\fl(\fa, t.f): t \in H\}
\end{align*}
of the fiber of the morphism $p$ over $\fa$ must be discrete.   For any $\lambda, \mu \in \Phi^+$ such that $\fg_{\lambda}, \fg_{\mu} \subseteq \fa$, choose nonzero elements $e_{\lambda} \in \fg_{\lambda}$ and $e_{\mu} \in \fg_{\mu}$.  Then we have
\begin{align*}
    (\exp(s.h).f)(e_{\lambda})(e_{\mu}) = e^{-s(\lambda(h) + \mu(h))} f(e_{\lambda})(e_{\mu})
\end{align*}
for all $s \in \mathbb C$ and $h \in \fh$.  From this we see that, for any $h \in \fh$ with $(\lambda + \mu)(h) \neq 0$, the subset $\{\fl(\fa, \exp(s.h).f): s\in \mathbb C\}$ is a curve connecting distinct points unless $f=0$.  Thus $f = 0$, as desired.
\end{proof}

From Proposition \ref{1} and Proposition \ref{2} we obtain

\begin{theorem} \label{lastthm}
Let $\fa$ be a Lie subalgebra of $\fg$ and $f \in Z^1_{CE}(\fa, \fad)$ a skew element.  Then the $D$-orbit $\Ad_D \fl (\fa, f)$ is closed if and only if $\fa$ is an abelian ideal of some Borel subalgebra of $\fg$ and $f = 0$.
\end{theorem}

Fix a Borel subalgebra $\fb$ of $\fg$.  

\begin{theorem}
The assignment $F: \fa \mapsto \Ad_D \fl(\fa, 0)$ provides a bijection between the set of abelian ideals of $\fb$ and the set of closed $D$-orbits in $\mathcal L$.  In particular, the number of closed $D$-orbits in $\mathcal L$ is exactly $2^{\rk \fg}$.
\end{theorem}

\begin{proof}
Let $\fa'$ be an abelian ideal of some Borel subalgebra $\fb'$ of $\fg$.  Since all Borel subalgebras of $\fg$ are conjugate, there exists $g \in G$ such that $\Ad_g \fb' = \fb$.  Then $\fa := \Ad_g \fa'$ is an abelian ideal of $\fb$.  It follows from Lemma \ref{halfact} that $$\Ad_g \fl(\fa', 0) = \fl(\Ad_g\fa', 0) = \fl(\fa, 0).$$  Hence, we have $$\Ad_D \fl(\fa', 0) = \Ad_D \fl(\fa, 0) = F(\fa).$$  This, together with Theorem \ref{lastthm}, proves that $F$ is surjective.

Let $\fa_1, \fa_2$ be abelian ideals of $\fb$.  Assume that $F(\fa_1) = F(\fa_2)$.  Again by Lemma \ref{halfact}, we can find $g \in G$ such that $\Ad_g \fl(\fa_1, 0) = \fl(\fa_2, 0)$, i.e., $\fl(\Ad_g \fa_1, 0) = \fl(\fa_2, 0)$.  It follows that $\Ad_g \fa_1 = \fa_2$.  Define $P_i$ to be the normalizer of $\fa_i$ in $G$, for $i = 1,2$.  It is easy to see that $gP_1g^{-1} = P_2$.  For $i = 1,2$, since $\fa_i$ is a Lie ideal of $\fb$, we see that $P_i$ contains the Borel subgroup $B$ corresponding to $\fb$.  But any parabolic subgroup of $G$ is conjugate to a unique parabolic subgroup containing $B$, hence we have $P_1 = P_2$.  Since the normalizer of a parabolic subgroup is itself, we deduce from $gP_1g^{-1} = P_2 = P_1$ that $g \in P_1$.  It follows that $\fa_1 = \Ad_g \fa_1 = \fa_2$.  This proves that $F$ is injective.

By Peterson's theorem on abelian ideals, the cardinality of the set of abelian ideals of $\fb$ is $2^{\rk \fg}$ (see \cite{CP0, CP, CP2, Kos2, Pan0, PR, Sut}).  Now the second statement follows easily.
\end{proof}

\bigskip


\noindent Sam Evens: Department of Mathematics, University of Notre Dame, 255 Hurley, Notre Dame, IN 46556.  Email: sevens@nd.edu

\medskip

\noindent Yu Li: Department of Mathematics, University of Chicago, 5734 S. University Ave., Chicago, IL 60637.  Email: liyu@math.uchicago.edu

\end{document}